\documentclass[11pt]{article}

\usepackage{vmargin,graphicx,amsmath,amssymb,color,subfigure,enumerate,csquotes,epstopdf}
\usepackage[utf8]{inputenc}
\usepackage{hyperref}

\newtheorem{theorem}{Theorem}[section]
\newtheorem{lemma}[theorem]{Lemma}
\newtheorem{notation}[theorem]{Notation}
\newtheorem{proposition}[theorem]{Proposition}

\def\blacksquare{
\thinspace\nobreak \vrule width 5pt height 5pt depth 0pt}
\newtheorem{remark}[theorem]{Remark}

\newenvironment{proof}{\begin{trivlist}
                       \item[]\hspace{0cm}{\bf Proof: }
                       \hspace{0cm} }{\hfill $\blacksquare$
                     \end{trivlist}}

\makeatletter

\@addtoreset{equation}{section}
\makeatother

\newcommand{\R}{\mathbb{R}}
\newcommand{\Z}{\mathbb{Z}}
\newcommand{\dx}{\mathrm{d}}
\newcommand{\Dom}{\mathsf{Dom}}
\newcommand{\ess}{\mathsf{ess}}
\newcommand{\dis}{\mathsf{dis}}
\newcommand{\Dir}{\mathsf{Dir}}
\newcommand{\Neu}{\mathsf{Neu}}

\newcommand{\Ai}{\mathsf{Ai}}

\newcommand{\red}{\mathsf{red1}}
\newcommand{\app}{\mathsf{red2}}
\newcommand{\mode}{\mathsf{mod}}
\newcommand{\spann}{\mathsf{span}}
\newcommand{\eps}{\varepsilon}

\newcommand{\tens}{\mathsf{tens}}
\DeclareMathOperator*{\esssup}{ess\,sup}

\title{Spectral asymptotics of a broken $\delta$-interaction}
\author{V. Duch\^ene\footnote{IRMAR, CNRS, Universit\'e de Rennes 1, Campus de Beaulieu, F-35042 Rennes cedex, France; 
e-mail: \texttt{vincent.duchene@univ-rennes1.fr}} \ and \ N.~Raymond\footnote{IRMAR, Universit\'e de Rennes 1, Campus de Beaulieu, F-35042 Rennes cedex, France;
e-mail: \texttt{nicolas.raymond@univ-rennes1.fr}}}

\begin{document}
\maketitle

\begin{abstract}
This paper is concerned with the spectral analysis of a Hamiltonian with a $\delta$-interaction supported along a broken line with angle $\theta$. The bound states with energy slightly below the threshold of the essential spectrum are estimated in the semiclassical regime $\theta\to 0$.
\end{abstract}

\section{Motivation and results}

\subsection{Motivation}

\subsubsection{Why breaking the $\delta$-interaction?}
The $\delta$-interaction supported on various geometries has attracted a lot of interest recently as an alternative to standard quantum graphs (see for instance \cite{BK13}). In particular the reader may consult the review by Exner~\cite{Exner08} for an introduction to leaky quantum graphs and the lecture notes by Post \cite{P12} for convergence results between the two objects. Our aim is to investigate the spectrum of a broken $\delta$-interaction. Before defining the main operator analyzed in this paper we shall present our initial motivation. 
In the paper by Exner and N{\v{e}}mcov{\'a}~\cite[Section 5]{EN03} (see also their related paper~\cite{EN01}) the authors were concerned by the existence and estimates of the discrete spectrum of a Hamiltonian with a $\delta$-interaction supported on a star. In particular they analyzed the simple case of a star with two branches in Section 5.2 for which their general result establishes the existence of discrete spectrum below the essential spectrum (see also~\cite{EI01} for the case when the $\delta$-interaction is supported on a curve). What's more is that they prove that the number of bound states tends to infinity when the angle between two branches of their stars is small: they even get an explicit lower bound (see~\cite[Remark 5.10]{EN03}). Moreover they also provide numerical simulations of the eigenvalues and eigenfunctions (see~\cite[Fig. 8 and Fig. 11]{EN03} and also~\cite[Fig. 1 and Fig. 4]{EN01}). The spectral behaviors which show up there should be compared with recent results about broken waveguides by Dauge, Lafranche and Raymond~\cite[Fig. 11]{DaLafRa11} and~\cite{DauRay12} where similar phenomena are observed. In this work, we will precisely quantify the number of eigenvalues generated by the breaking of the support of a $\delta$-interaction and provide their asymptotic expansions when the breaking is strong (such spectral questions are quite natural as we can see in the related works~\cite{EJ13} and~\cite{Kondej13}). We will complete the considerations of~\cite{EN03} (and also~\cite{BEW08}) when the number of branches is two thanks to the light of semiclassical analysis. At the same time the present paper will provide some insight into Open Problem 7.3 in~\cite{Exner08}.

\subsubsection{Definition of the main operator}
Let us now define our main operator. For $\alpha>0$, we introduce the following quadratic form
\begin{equation}\label{Q}
\mathcal{Q}_{\theta,\alpha}(\psi)=\int_{\R^2} |\nabla\psi|^2\dx u \dx v-\alpha\int_{\R} |\psi(|s|\cos\theta,s\sin\theta)|^2\dx s,\quad\forall\psi\in H^1(\R^2),
\end{equation}
where $\theta\in\left(0,\frac{\pi}{2}\right)$ is the breaking angle. This is well-known that $\mathcal{Q}_{\theta,\alpha}$ is semi-bounded (see~\cite{BEKS94}). In particular we may consider its Friedrichs extension $\mathcal{H}_{\theta,\alpha}$. We can formally write
$$\mathcal{H}_{\theta,\alpha}=-\Delta-\alpha\delta_{\Sigma_{\theta}},$$
where 
$$\Sigma_{\theta}=\{(|s|\cos\theta, s\sin\theta),\quad s\in\R\}.$$
The following characterization of the essential spectrum is well-known (see~\cite{EI01}).
\begin{lemma}\label{essential}
We have
$$\sigma_{\ess}(\mathcal{H}_{\theta,\alpha})=\left[-\frac{\alpha^2}{4},+\infty\right).$$
\end{lemma}

We would like to describe the spectrum below the essential spectrum in the strong breaking limit $\theta\to 0$. For that purpose we shall perform the following rescaling:
\begin{equation}\label{rescaling}
x=\alpha\frac{\sin\theta}{\cos^2\theta} u,\quad y=\alpha\frac{1}{\cos\theta}v,
\end{equation}
which permits to rephrase the problem into a semiclassical problem. 
We introduce the unitary transform, defined for $\psi\in L^2(\R^2)$ by
$$U_{\theta,\alpha}\psi(x,y)=\frac{\cos^{3/2}\theta}{\alpha\sin^{1/2}\theta}\psi\left(\frac{\cos^2\theta}{\alpha\sin\theta} x,\frac{\cos\theta}{\alpha}y\right).$$
We have $\mathcal{H}_{\theta,\alpha}=\alpha^2(1+h^2)U_{\theta,\alpha}^{-1}\mathfrak{H}_{h}U_{\theta,\alpha}$  where $\mathfrak{H}_{h}$ is the Friedrichs extension of the rescaled quadratic form:
\begin{equation}\label{rQ}
\mathfrak{Q}_{h}(\psi)=\int_{\R^2} h^2|\partial_{x}\psi|^2+|\partial_{y}\psi|^2 \dx x \dx y-\int_{\R} |\psi(|s|,s)|^2\dx s,\quad\forall\psi\in H^1(\R^2),
\end{equation}
and where $h=\tan\theta$. Formally we may write
\begin{equation}\label{main-op}
\mathfrak{H}_{h}=-h^2\partial_{x}^2-\partial^2_{y}-\delta_{\Sigma_{\frac{\pi}{4}}}.
\end{equation}
In particular, we notice that:
$$\sigma_{\ess}(\mathfrak{H}_{h})=\left[-\frac{1}{4(1+h^2)},+\infty\right).$$
\begin{notation}\label{lan}
We denote by $\lambda_{n}(h)$ the $n$-th eigenvalue, if it exists, of $\mathfrak{H}_{h}$. More generally for a semi-bounded quadratic form $\mathfrak{Q}_{h}^{\natural}$, we denote by $\mathfrak{H}_{h}^{\natural}$ the corresponding Friedrichs extension and by $\lambda_{n}^{\natural}(h)$ the $n$-th eigenvalue, if it exists.
Let us also recall the min-max characterization of the $n$-th eigenvalue. We have
\[
   \lambda^\natural_n(h) =
   \inf_{\substack{G\subset \Dom(\mathfrak{Q}^\natural_{h})\\ \dim G=n}} \ \sup_{\substack{
    \psi\in G}}   \frac{\mathfrak{Q}^\natural_{h}(\psi)}{\Vert\psi\Vert^2} \,.
\]
Here and below $\|\cdot\|$ denotes the standard $L^2$ norm on $\R^2$ while $\|\cdot\|_{L^2(\R)}$ is the $L^2$ norm on $\R$. 
We will also denote by $\langle\cdot,\cdot\rangle_{L^2(\R_{y})}$ the partial scalar product defined by:
$$\langle\psi_{1},\psi_{2}\rangle_{L^2(\R_{y})}=\int_{\R_{y}}\psi_{1}(x,y)\psi_{2}(x,y)\dx y$$
and by $\|\cdot\|_{L^2(\R_{y})}$ the corresponding norm.
\end{notation}
By using this semiclassical reformulation we will easily get an explicit lower bound for $\mathcal{Q}_{\theta,\alpha}$.
\begin{proposition}\label{semi-bounded}
For all $\psi\in H^1(\R^2)$ and $\theta\in\left(0,\frac{\pi}{2}\right)$:
$$\mathcal{Q}_{\theta,\alpha}(\psi)\geq -\frac{\alpha^2}{\cos^2\theta}\Vert\psi\Vert^2.$$
\end{proposition}
\begin{remark}
In fact this lower bound permits to define directly the Friedrichs extension associated with $\mathcal{H}_{\theta,\alpha}$ without using the general result of~\cite{BEKS94}. This lower bound degenerates when $\theta$ goes to $\frac{\pi}{2}$ but, as we will see, it is more and more accurate when $\theta$ goes to $0$. A fine lower bound (independently from $\theta$) is obtained in~\cite{L13}. In the regime $\theta\to 0$, an easy corollary of one of our main results will provide a description of the optimal lower bound.
\end{remark}

\subsection{Main results and organization of the paper}
Let us now state the main results of this paper. Our first result is an estimate of the number of eigenvalues of $\mathfrak{H}_{h}$ below the threshold of the essential spectrum. For this purpose we shall introduce some notation.
\begin{notation}\label{Lambert}
We denote by $W : [-e^{-1},+\infty) \to [-1,+\infty)$ the Lambert function defined as the inverse of $[-1,+\infty)\ni w\mapsto we^w\in[-e^{-1},+\infty)$.
\end{notation}
\begin{notation}\label{N.number}
Given $\mathfrak H$ a semi-bounded self-adjoint operator and $a<\inf\sigma_{\ess}(\mathfrak H)$, we denote
\[ \mathcal N(\mathfrak H,a)\ = \ \#\{ \lambda\in \sigma(\mathfrak H) \ : \ \lambda \leq a \}<+\infty .\]
The eigenvalues are counted with multiplicity.
\end{notation}

\begin{theorem}\label{number}
There exists $M_{0}>0$ such that for all $C(h)\geq M_{0}h$ with $C(h)\underset{h\to 0}{\to} C_0\geq 0$:
$$\mathcal{N}\left(\mathfrak{H}_{h},-\frac{1}{4}-C(h)\right)\underset{h\to 0}{\sim}\frac{1}{\pi h}\int_{x=0}^{+\infty}\sqrt{\left(-\frac{1}{4}-C_0+\left(\frac{1}{2}+\frac{1}{2x}W(xe^{-x})\right)^2\right)_{+}}\dx x,$$
with the notation $f_+(x)=\max\{0,f(x)\}$.
\end{theorem}
\begin{remark}
It is important to notice that in the above result, we estimate the counting function below a potentially moving (w.r.t. $h$) threshold. In particular, the distance between $-\frac{1}{4}-C(h)$ and the bottom of the essential spectrum is allowed to vanish in the semiclassical limit. Therefore our statement is slightly unusual as customary results would typically concern $\mathcal{N}\left(\mathfrak{H}_{h},E\right)$ with $E$ fixed and satisfying $E<-\frac{1}{4}$, so as to insure a fixed security distance to the bottom of the essential spectrum (see for instance the related work~\cite{MoTr05}). 
\end{remark}
\begin{remark}
In the small angle limit, this result is a refinement of~\cite[Remark 5.10]{EN03}. Indeed, Exner and N\v{e}mcov\'a show that the number of bound states grows as $n\gtrsim  \frac{C}{\pi h}$ with $C=\frac{3^{3/2}}{8\sqrt{5}}\approx 0.290$, whereas our result implies a better constant $C\approx 1.379$.
\end{remark}
Our second result concerns the asymptotics of the low lying spectrum of $\mathfrak{H}_{h}$. Let us first recall the definition of the Airy operator.
\begin{notation}\label{Airy}
The Airy operator is the Dirichlet realization on $L^2((0,+\infty))$ of $-\partial_{x}^2+x$. Its $n$-th eigenvalue is nothing but the absolute value of the $n$-th zero (counted in decreasing order), denoted by $z_{\Ai}(n)$, of the standard Airy function.
\end{notation}
\begin{theorem}\label{firsteigenvalues}
For all $n\geq 1$, we have:
$$\lambda_{n}(h)\underset{h\to 0}{=}-1+2^{2/3}z_{\Ai}(n) h^{2/3}+O(h).$$
\end{theorem}
\begin{remark}
This asymptotic expansion explains the behavior of the spectral curves of~\cite[Fig. 8]{EN03} when the angle approaches zero: the behavior of the first eigenvalues is governed by the Airy operator. Our result is a refinement (in the small angle limit) of~\cite{BEW08} since we have an accurate description of the first eigenvalues and not only an upper bound of the first one (see also the upper bounds of the first eigenvalue obtained in~\cite{BEW09} for star graphs).
\end{remark}
From Theorem~\ref{firsteigenvalues} this is possible to deduce a quasi-tensorial structure of the first eigenfunctions.
\begin{theorem}\label{tensorisation-x=0}
For all $C_{0}>0$, there exist $h_{0}>0$, $C>0$ such that for all $h\in(0,h_{0})$ and all eigenpairs $(\lambda,\psi)$ such that $\lambda\leq -1+C_{0}h^{2/3}$, we have
$$\int_{\R^2} |\psi-\Pi_{0}\psi|^2\dx x\dx y\leq Ch^{2/3}\Vert\psi\Vert^2,$$
where $\Pi_{0}\psi=\langle\psi,e^{-|y|} \rangle_{L^2(\R_{y})}e^{-|y|}$.
\end{theorem}

\paragraph{Remarks on $\delta$-interactions on crossing lines} Let us consider, as in \cite{L13} and after the rescaling \eqref{rescaling}, the following quadratic form, defined for $\psi\in H^1(\R^2)$ by
$$\mathfrak{Q}^{\times}_{h}(\psi)=\int_{\R^2}h^2|\partial_{x}\psi|^2+|\partial_{y}\psi|^2\dx x\dx y-\int_{\R}\left(|\psi(-s,s)|^2+|\psi(s,s)|^2\right)\dx s.$$
The strategy of our proofs can apply modulo straightforward modifications and we get the following asymptotics
$$\mathcal{N}\left(\mathfrak{H}^\times_{h},-\frac{1}{4}-C(h)\right)\underset{h\to 0}{\sim}\frac{2}{\pi h}\int_{x=0}^{+\infty}\sqrt{\left(-\frac{1}{4}-C_0+\left(\frac{1}{2}+\frac{1}{2x}W(xe^{-x})\right)^2\right)_{+}}\dx x.$$
In the same way, we have, for $n\geq 1$,
$$\lambda_{2n}^\times(h)=-1+2^{2/3}z_{\Ai}(n) h^{2/3}+O(h),$$
and 
$$\lambda_{2n-1}^\times(h)=-1+2^{2/3}z_{\Ai'}(n) h^{2/3}+O(h),$$
where $z_{\Ai'}(n)$ is the absolute value of the $n$-th zero (counted in decreasing order) of the derivative of the Airy function.

\paragraph{Philosophy of the proofs} Let us now discuss the general philosophy of the proofs. As suggested by the expression~\eqref{main-op}, the main ingredient in this paper is a dimensional reduction in the spirit of the famous Born-Oppenheimer approximation (see~\cite{BO27, Martinez89, KMSW92}). Such dimensional reductions where used by Balazard-Konlein in~\cite{B85} in a pseudo-differential context (and thus in a very regular framework) to estimate numbers of eigenvalues. Let us also mention the paper by Morame and Truc~\cite{MoTr05} where this kind of questions appears (with a regular electric potential). It turns out that our framework is strongly excluded by the assumptions of~\cite{B85} since the $\delta$-interaction is not even an electric potential. Nevertheless we will see that a pure variational analysis can overturn this difficulty.

\paragraph{Organization of the paper} This paper is organized as follows. In Section~\ref{S2} we introduce the double $\delta$-well in dimension one and we recall their basic spectral properties. In particular we will prove Proposition~\ref{semi-bounded}. Section~\ref{S3} is devoted to the dimensional reduction of $\mathfrak{H}_{h}$ to model operators in dimension one (see Proposition~\ref{comp}). Finally Section~\ref{S4} is concerned with the analysis of one-dimensional operators and with the proof of Theorems~\ref{number},~\ref{firsteigenvalues} and~\ref{tensorisation-x=0}.

\section{Double $\delta$-well}\label{S2}
For $x\geq 0$, we introduce the quadratic form $\mathfrak{q}_{x}$ defined for $\psi\in H^1(\R)$ by
\begin{equation}\label{qx}
\mathfrak{q}_{x}(\psi)=\int_{\R} |\psi'(y)|^2\dx y-|\psi(-x)|^2-|\psi(x)|^2.
\end{equation}
This is standard (see~\cite[Chapter II.2]{AGHH88} and also~\cite{BEKS94}) that $\mathfrak{q}_{x}$ is a semi-bounded and closed quadratic form on $H^1(\R)$. Therefore we may introduce the associated self-adjoint operator denoted by $\mathfrak{D}_{x}$ whose domain is
$$\Dom(\mathfrak{D}_{x})=\left\{\psi\in H^1(\R)\cap H^2(\R\setminus\{\pm x\}) : \lim_{\eps\searrow 0}\left(\psi'(\pm x+\eps)-\psi'(\pm x-\eps)\right)=-\psi(\pm x)\right\}$$
and defined as $\mathfrak{D}_{x}\psi(y)=-\psi''(y)$. We can write formally
$$\mathfrak{D}_{x}=-\partial_{y}^2-\delta_{-x}-\delta_{x}.$$
Let us describe the spectrum of $\mathfrak{D}_{x}$. The following lemma is obvious.
\begin{lemma}
For all $x\geq 0$, the essential spectrum of $\mathfrak{D}_{x}$ is given by
$$\sigma_{\ess}(\mathfrak{D}_{x})=[0,+\infty).$$
\end{lemma}

\begin{notation}
For $x\geq 0$, we denote by $\mu_{1}(x)$ the lowest eigenvalue of $\mathfrak{D}_{x}$ and by $u_{x}$ the corresponding positive and $L^2$-normalized eigenfunction.
\end{notation}
In fact we can give an explicit expression of the pair $(\mu_{1}(x),u_{x})$. The following proposition is essentially well-known, except maybe its last two points.
\begin{proposition}\label{delta1D}
For $x\geq 0$, we have
$$\mu_{1}(x)=-\left(\frac{1}{2}+\frac{1}{2x}W(xe^{-x})\right)^2.$$
The second eigenvalue $\mu_{2}(x)$ only exists for $x>1$ and is given by
$$\mu_{2}(x)=-\left(\frac{1}{2}+\frac{1}{2x}W(-xe^{-x})\right)^2.$$
By convention we set $\mu_{2}(x)=0$ when $x\leq 1$.
In particular we have the following properties (see illustration in Figure~\ref{F.mu}):
\begin{enumerate}
\item\label{1} $\mu_{1}(x)\underset{x\to 0}{=}-1+2x+O(x^2)$,
\item\label{2} $\mu_{1}(x)\underset{x\to+\infty}{=}-\frac{1}{4}-\frac{e^{-x}}{2}+O(xe^{-2x})$, $\mu_{2}(x)\underset{x\to+\infty}{=}-\frac{1}{4}+\frac{e^{-x}}{2}+O(xe^{-2x})$,
\item\label{3} For all $x\geq 0$, $-1\leq \mu_{1}(x)<-\frac{1}{4}$ and for all $x>1$, $\mu_{2}(x)>-\frac{1}{4}$,
\item\label{4} $\mu_{1}$ admits a unique minimum at $0$,
\item\label{5} For all $x\geq 0$ and all $f \in H^1(\R)$, we have $\mathfrak{q}_{x}(f)\geq -\Vert f\Vert_{L^2(\R)}^2$,
\item\label{6} $R(x):=\Vert\partial_{x}u_{x}\Vert^2_{L^2(\R_{y})}$ defines a bounded function for $ x> 0$.
\item\label{7} $\Vert\partial_{y}u_{x}\Vert^2_{L^2(\R_{y})}$ defines a bounded function for $x\geq 0$. 
\end{enumerate}
\end{proposition}
\begin{figure}
\includegraphics[width=\textwidth]{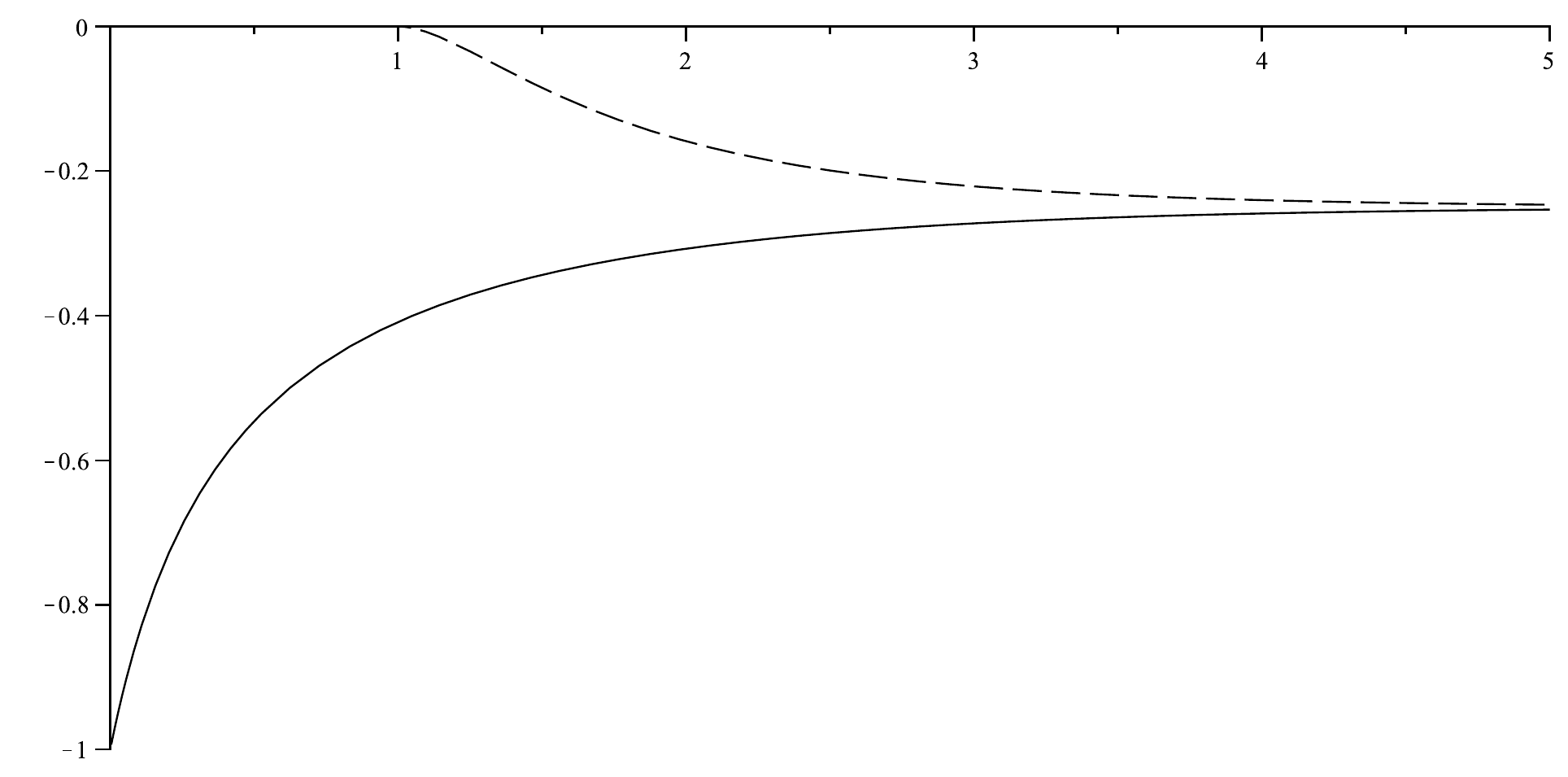}
\caption{The eigenvalues of $\mathfrak{D}_{x}$ as functions of $x$: $\mu_1(x)$ (solid) and $\mu_2(x)$ (dashed).}
\label{F.mu}
\end{figure}
\begin{proof}
Let us solve the eigenvalue equation
$$\mathfrak{D}_{x}\psi_x=-\lambda_x\psi_x.$$
Up to multiplicative constants and using the continuity of the elements of $\Dom(\mathfrak{D}_{x})$ we have the alternative
\[ \psi_x=\psi_{x,1}\text{ or } \psi=\psi_{x,2}.\]
where
\[ \psi_{x,1}(y)=\begin{cases}
e^{\sqrt{\lambda_x}(x+y)} \text{ if } y\leq-x \\
\frac{1}{\cosh(\sqrt{\lambda_x} x)}\cosh(\sqrt{\lambda_x} y)  \text{ if } -x<y<x\\
e^{\sqrt{\lambda_x}(x-y)} \text{ if } y\geq x
\end{cases}\]
and
\[\psi_{x,2}(y)=\begin{cases}
e^{\sqrt{\lambda_x}(x+y)} \text{ if } y\leq-x \\
\frac{-1}{\sinh(\sqrt{\lambda_x} x)}\sinh(\sqrt{\lambda_x} y)  \text{ if } -x<y<x\\
-e^{\sqrt{\lambda_x}(x-y)} \text{ if } y\geq x
\end{cases}.\]
In the case $\psi=\psi_{x,1}$, the condition at $\pm x$ becomes
\[\left( 2\sqrt\lambda_{x,1}-1\right)e^{2\sqrt{\lambda_{x,1}} x}=1,\]
and we see that $\sqrt{\lambda_{x,1}}\geq\frac{1}{2}$.
In terms of the Lambert function, we have
$$\sqrt{\lambda_{x,1}}=\frac{1}{2}+\frac{1}{2x}W(xe^{-x})=:\sqrt{-\mu_1(x)}.$$
In the case $\psi=\psi_{x,2}$ we find in the same way, for $x>1$,
$$\sqrt{\lambda_{x,2}}=\frac{1}{2}+\frac{1}{2x}W(-xe^{-x})=:\sqrt{-\mu_2(x)}.$$
Recall that for $x\in(0,1]$, we set $\sqrt{\lambda_{x,2}}=0$. In addition, we find $\sqrt{\lambda_{x,2}}\leq\frac{1}{2}$ for $x>1$.

This is very standard to establish the points~\ref{1},~\ref{2},~\ref{3}. For the point~\ref{4}, we notice that $\mu_{1}(x)=-1$ for $x>0$ is equivalent to $W(xe^{-x})=x$ which admits no solution for $x>0$. The point~\ref{5} is then obvious.

Let us now prove the point~\ref{6}. We notice that $\psi_{x,1}(y)$ can be rewritten in the form
\begin{multline*}
\psi_{x,1}(y)=H(-x-y)e^{\sqrt{-\mu_1(x)}(x+y)}+H(-x+y)e^{\sqrt{-\mu_1(x)}(x-y)} \\
+ H(x+y)H(x-y)\frac{\cosh(\sqrt{-\mu_1(x)}\, y) }{\cosh(\sqrt{-\mu_1(x)}\, x)},
\end{multline*}
where $H(\cdot)$ is the Heaviside function (with $H(0)=\frac{1}{2}$).
Now, one easily checks that
\[ 0\leq \frac{\cosh(\sqrt{-\mu_1(x)}\, y) }{\cosh(\sqrt{-\mu_1(x)}\, x)} \leq e^{-\sqrt{-\mu_1(x)}(x+y)}+e^{\sqrt{-\mu_1(x)}(y-x)},\]
so that
\[   H(-x-y)e^{\sqrt{-\mu_1(x)}(x+y)}+H(-x+y)e^{\sqrt{-\mu_1(x)}(x-y)}  \leq \psi_{x,1}(y)\leq e^{-\sqrt{-\mu_1(x)}|x+y|}+e^{-\sqrt{-\mu_1(x)}|y-x|} ,\]
and therefore there exist positive constants $c,C$, independent of $x$, such that
\[ 0<c \leq \big\Vert \psi_{x,1}\big\Vert_{L^2(\R_{y})}\leq C<\infty .\]
In the same way, one can check the following estimates:
\begin{align*}
\left|\partial_{x}\left(e^{\sqrt{-\mu_1(x)}(x+y)}\right)\right|\leq \left(|(\sqrt{-\mu_1})'(x)||x+y|+\sqrt{-\mu_1(x)}\right)e^{-\sqrt{-\mu_1(x)}|x+y|},\forall y\leq-x\\
\left|\partial_{x}\left(e^{\sqrt{-\mu_1(x)}(x-y)}\right)\right|\leq  \left(|(\sqrt{-\mu_1})'(x)||x-y|+\sqrt{-\mu_1(x)}\right)e^{-\sqrt{-\mu_1(x)}|x-y|},\forall y\geq x
\end{align*}
and, for $y\in[-x,x]$, 
\[\left|\partial_{x}\left(\frac{\cosh(\sqrt{-\mu_1(x)}y)}{\cosh(\sqrt{-\mu_1(x) x})}\right)\right|\leq \left(e^{-\sqrt{-\mu_1(x)}(x+y)}+e^{\sqrt{-\mu_1(x)}(y-x)}\right)\left(2x|(\sqrt{-\mu_1})'(x)|+\sqrt{-\mu_1(x)}\right).\]
Therefore, we deduce that there exists a positive constant, $C'$, independent of $x$, such that
\[  \big\Vert \partial_x\psi_{x,1} \big\Vert^2_{L^2(\R_{y})}\leq C'<\infty .\]
By definition, $u_x(y)=\frac{\psi_{x,1}(y)}{\Vert\psi_{x,1} \Vert_{L^2(\R_{y})} }$. It follows by elementary computations that
\[R(x)\leq 4\frac{\Vert\partial_{x}\psi_{x,1}\Vert^2_{L^2(\R_{y})}}{\Vert \psi_{x,1} \Vert_{L^2(\R_{y})}^2} ,\]
and the point $6$ is proved.

Finally, one obtains the point $7$ by remarking
\[ \Vert\partial_{y}u_{x}\Vert^2_{L^2(\R_{y})} \ = \ \mu_1(x)+|u_x(x)|^2+|u_x(-x)|^2 \ = \ \mu_1(x)+\frac2{\Vert \psi_{x,1} \Vert_{L^2(\R_{y})}^2}.\]
\end{proof}
As a direct application of Proposition~\ref{delta1D}, we have
\begin{proposition}\label{lb-rQ}
For all $\psi\in H^1(\R^2)$ and for all $h>0$, we have:
$$\mathfrak{Q}_{h}(\psi)\geq\int_{\R_{x}}\int_{\R_{y}} \left(h^2|\partial_{x}\psi|^2+\tilde\mu_{1}(x)|\psi|^2\right)\dx y\dx x,$$
where $\tilde\mu_{1}(x)=\mu_{1}(x)$, for $x\geq 0$ and $0$ elsewhere. In particular, we have:
$$\mathfrak{Q}_{h}(\psi)\geq-\Vert\psi\Vert^2$$
or equivalently, for all $\theta\in\left(0,\frac{\pi}{2}\right)$:
$$\mathcal{Q}_{\theta,\alpha}(\psi)\geq -\frac{\alpha^2}{\cos^2\theta}\Vert\psi\Vert^2.$$
\end{proposition}
\begin{proof}
For $\psi\in H^1(\R^2)$, we have:
$$\mathfrak{Q}_{h}(\psi)=\int_{\R^2} h^2|\partial_{x}\psi|^2+|\partial_{y}\psi|^2 \dx x \dx y-\int_{\R} |\psi(|s|,s)|^2\dx s$$
so that:
\begin{multline*}
\mathfrak{Q}_{h}(\psi)=\int_{x\in\R^{+}}\left(\int_{\R_{y}} h^2|\partial_{x}\psi|^2\dx y+\int_{\R_{y}}|\partial_{y}\psi|^2\dx y - |\psi(-x,x)|^2- |\psi(x,x)|^2\right)\dx x\\
+\int_{x\in\R^{-}}\int_{\R_{y}} h^2|\partial_{x}\psi|^2+|\partial_{y}\psi|^2\dx y\dx x.
\end{multline*}
We infer that:
\begin{equation*}
\mathfrak{Q}_{h}(\psi)\geq\int_{x\in\R^{+}}\int_{\R_{y}} \left(h^2|\partial_{x}\psi|^2+\mu_{1}(x)|\psi|^2\right)\dx y\dx x+\int_{x\in\R^{-}}\int_{\R_{y}} h^2|\partial_{x}\psi|^2\dx y\dx x,
\end{equation*}
and the conclusions follow.
\end{proof}

\section{Spectral reductions}\label{S3}
Now we would like to use the spectral theory of $\mathfrak{D}_{x}$ in order to compare the operator $\mathfrak{H}_{h}$ with simpler operators. 
\subsection{Dimensional reduction}
In order to deal with the singularity at $x=0$, we introduce the following extension of $u_{x}$.
\begin{notation}
Let us define
$$\tilde u_{x}(y)=\begin{cases}
u_{x}(y) \text{ if } x\geq 0\\
u_{0}(y)  \text{ if } x<0\\
\end{cases}.$$
We also introduce the projections defined for $\psi\in L^2(\R^2)$ by
$$\Pi\psi(x,y)=\langle\psi,\tilde u_{x}\rangle_{L^2(\R_{y})}\tilde u_{x}(y),\quad \Pi^\perp\psi(x,y)=\psi(x,y)-\Pi\psi(x,y).$$
If $\varphi=\varphi(x,y)$, we denote $\varphi_{x}(y)=\varphi(x,y)$.
\end{notation}

\begin{lemma}\label{op-red}
For all $\psi\in\Dom(\mathfrak{Q}_{h})$, the function $\Pi\psi$ belongs to $\Dom(\mathfrak{Q}_{h})$ and we have
$$\mathfrak{Q}_{h}(\Pi\psi)=\int_{\R_{x}} h^2|f'(x)|^2+(\hat\mu_{1}(x)+h^2\tilde R(x))|f(x)|^2  \dx x,\quad \mbox{ with } f(x)=\langle\psi, \tilde u_{x}\rangle_{L^2(\R_{y})},$$
where $\hat\mu_{1}(x)=\mu_{1}(x)$ for $x\geq 0$ and $\hat\mu_{1}(x)=1$ for $x<0$ and $\tilde R(x)=R(x)$ for $x> 0$ and $\tilde R(x)=0$ for $x\leq 0$.
\end{lemma}
\begin{proof} Recall that $\Dom(\mathfrak{Q}_{h})=H^1(\R^2)$. By Proposition~\ref{delta1D}, one has
\[  \esssup_{x\in\R} \Vert\partial_{x}\tilde u_{x}\Vert^2_{L^2(\R_{y})} \ = \ \sup_{x>0} R(x) \ < \ \infty \]
and
\[  \esssup_{x\in\R} \Vert\partial_{y}\tilde u_{x}\Vert^2_{L^2(\R_{y})} \ = \ \sup_{x\geq 0} \Vert\partial_{y}\tilde u_{x}\Vert^2_{L^2(\R_{y})}  \ < \ \infty.\]
It follows immediately that, for any $\psi\in H^1(\R^2)$,
\[ \partial_x \big(\Pi\psi) \ = \ f(x)\partial_x \tilde u_{x}(y) +f'(x)\tilde u_{x}(y)\in L^2(\R^2),\]
since $\esssup_{x\in\R}  f'(x)\leq\esssup_{x\in\R}  \langle\psi, \partial_x \tilde u_{x}\rangle_{L^2(\R_{y})}+\esssup_{x\in\R}  \langle \partial_x \psi, \tilde u_{x}\rangle_{L^2(\R_{y})}<\infty$, and
\[ \partial_y \big(\Pi\psi) \ = \ f(x)\partial_y \tilde u_{x}(y) \in L^2(\R^2).\]
Thus one has $\Pi\psi\in H^1(\R^2)=\Dom(\mathfrak{Q}_{h})$, and the calculations thereafter are valid. By definition, one has
\begin{align*}
\mathfrak{Q}_{h}(\Pi\psi)&=\int_{\R^2}\!\! h^2 |f(x)\partial_x \tilde u_{x}(y) +f'(x)\tilde u_{x}(y)|^2 +|f(x)|^2|\partial_y \tilde u_{x}(y)|^2\dx x\dx y  -\int_{\R}|f(|s|)\tilde u_{|s|}(s)|^2 \dx s\\
&=\int_{\R_x} h^2 |f'(x)|^2+h^2|f(x)|^2\Vert \partial_x \tilde u_{x}(y)\Vert_{L^2(\R_y)}^2\dx x +\int_{\R_x^-}|f(x)|^2\int_{\R_y}|\partial_y \tilde u_{x}(y)|^2\dx y \dx x\\
&\qquad +\int_{\R_x^+}|f(x)|^2\int_{\R_y}|\partial_y \tilde u_{x}(y)|^2- \left(|\tilde u_{x}(-x)|^2+|\tilde u_{x}(x)|^2\right) \dx x\\
&  =\int_{\R_{x}} h^2|f'(x)|^2\dx x+\int_{\R_{x}^+} h^2  |f(x)|^2R(x) \dx x+\int_{\R_{x}^-} |f(x)|^2  \dx x+\int_{\R_{x}^+} |f(x)|^2 \mu_{1}(x) \dx x ,
\end{align*}
where we used Fubini's theorem, and the following properties on $\tilde u_{x}(y)$:
\begin{itemize}
\item $\forall x\in \R$, $\tilde u_{x}$ is normalized in $L^2(\R_y)$, and in particular, for any $x\neq 0$, 
\[ 2\langle \tilde u_{x},\partial_x \tilde u_{x}\rangle_{L^2(\R_{y})}=\frac{\dx}{\dx x}\langle \tilde u_{x},\tilde u_{x}\rangle_{L^2(\R_{y})} =0.\]
\item $\forall x>0$, one has $\mathfrak{q}_{x}(\tilde u_{x})=\mu_1(x)$.
 \item $\forall x\leq 0$, one has $\int_{\R_y}|\partial_y \tilde u_{x}(y)|^2\dx y=\int_{\R_y}|\partial_y  u_{0}(y)|^2\dx y=1$.
\end{itemize}
The result is now straightforward.
\end{proof}
We get the same result for the corresponding bilinear form $\mathfrak{B}_{h}$.
\begin{lemma}\label{op-red-bil}
For all $\psi_{1},\psi_{2}\in\Dom(\mathfrak{Q}_{h})$, we have
$$\mathfrak{B}_{h}(\Pi\psi_{1},\Pi\psi_{2})=\int_{\R_{x}} h^2f_{1}'(x)f'_{2}(x)+(\hat\mu_{1}(x)+h^2\tilde R(x))f_{1}(x) f_{2}(x)  \dx x,$$
with $f_{j}(x)=\langle\psi_{j}, \tilde u_{x}\rangle_{L^2(\R_{y})}$.
\end{lemma}
Let us now use the orthogonal decomposition to bound $\mathfrak{Q}_{h}$ from below.
\begin{proposition}\label{dimensional-reduction-lb}
For all $\psi\in\Dom(\mathfrak{Q}_{h})$ and all $\eps\in(0,1)$, we have
\begin{multline*}
\mathfrak{Q}_{h}(\psi)\geq \int_{\R_x} (1-\eps)h^2|f'(x)|^2+\big(\tilde\mu_1(x)-4\eps^{-1} h^2 \tilde R(x)\big)|f(x)|^2 \dx x\\
+\int_{\R_x} (1-\eps)h^2\Vert\partial_{x}\Pi^\perp\psi\Vert_{L^2(\R_{y})}^2+\big(\tilde\mu_2(x)-4\eps^{-1} h^2 \tilde R(x)\big)\Vert\Pi^\perp\psi\Vert_{L^2(\R_y)}^2 \dx x,
\end{multline*}
where $\tilde\mu_{i}(x)=\mu_{i}(x)$ for $x\geq 0$ and $\tilde\mu_{i}(x)=0$ for $x<0$ ($i\in\{1,2\}$); $\tilde R(x)=R(x)$ for $x>0$ and $\tilde R(x)=0$ for $x\leq 0$.
\end{proposition}
\begin{proof} 
By definition, one has for any $\psi\in\Dom(\mathfrak{Q}_{h})=H^1(\R^2)$,
\[
\mathfrak{Q}_{h}(\psi)=\int_{\R^2} h^2 |\partial_x \psi |^2\dx x\dx y+\int_{\R_x^-\times \R_y}|\partial_y \psi |^2\dx x\dx y+\int_{\R_x^+} \mathfrak{q}_{x}(\psi_x) \dx x.\]
Since $\psi\in\Dom(\mathfrak{Q}_{h})=H^1(\R^2)$, one has $\Pi\psi\in H^1(\R^2)$ and $\Pi^\perp\psi=\psi-\Pi\psi\in H^1(\R^2)$. Moreover, for any fixed $x\geq 0$, recall that $u_{x}$ is an eigenfunction corresponding to an eigenvalue of $\mathfrak D_x$, thus one has
\[ \forall x\geq 0, \quad  \mathfrak{q}_{x}(\psi_x) \ = \ \mathfrak{q}_{x}((\Pi\psi)_{x})+\mathfrak{q}_{x}((\Pi^\perp\psi)_{x}) \ \geq \ \mu_1(x) \Vert \Pi\psi \Vert_{L^2(\R_y)}^2+ \mu_2(x) \Vert \Pi^\perp \psi \Vert_{L^2(\R_y)}^2,\]
where we have applied the min-max principle to the quadratic form $\mathfrak{q}_{x}$ and to the functions $(\Pi\psi)_{x}$ and $(\Pi^\perp\psi)_{x}$ which are orthogonal in $L^2(\R_{y})$.

Now, one has $\langle \Pi\varphi,\Pi^\perp\varphi\rangle_{L^2(\R_y)}=0$, for any $\varphi\in L^2(\R^2)$, therefore
\begin{align*} \Vert\partial_x \psi \Vert_{L^2(\R_y)}^2 \ &= \ \Vert \Pi \partial_x \psi \Vert_{L^2(\R_y)}^2+\Vert\Pi^\perp \partial_x \psi \Vert_{L^2(\R_y)}^2\\
&= \ \Vert \partial_x(\Pi \psi) - \mathcal R(x,y)\Vert_{L^2(\R_y)}^2+\Vert \partial_x(\Pi^\perp \psi) +\mathcal R(x,y) \Vert_{L^2(\R_y)}^2,
\end{align*}
with
\[ \mathcal R(x,y) \ := \ \big[\partial_x ,\Pi\big]\psi \ = \ \langle\psi,\partial_x \tilde u_{x}\rangle_{L^2(\R_{y})}\tilde u_{x}(y)+\langle\psi,\tilde u_{x}\rangle_{L^2(\R_{y})}\partial_x \tilde u_{x}(y).\]
It follows, for all $\eps\in(0,1)$,
\begin{multline*} \Vert\partial_x \psi \Vert_{L^2(\R_y)}^2 \ \geq \ (1-\eps)\Vert \partial_x(\Pi \psi) \Vert_{L^2(\R_y)}^2+(1-\eps)\Vert \partial_x(\Pi^\perp \psi) \Vert_{L^2(\R_y)}^2\\
-2(\eps^{-1}-1)\Vert \mathcal R(x,y)\Vert_{L^2(\R_y)}^2.
\end{multline*}

Now, notice, for any $x>0$,
\[ \Vert \mathcal R(x,y)\Vert_{L^2(\R_y)}^2 \ = \ \langle\psi,\partial_x \tilde u_{x}\rangle_{L^2(\R_{y})}^2+\langle\psi,\tilde u_{x}\rangle_{L^2(\R_{y})}^2 \Vert \partial_x \tilde u_{x}\Vert_{L^2(\R_y)}^2 \ \leq \ 2 \  R(x)\ \Vert \psi_x \Vert_{L^2(\R_y)}^2,\]
where we used Proposition~\ref{delta1D}; and for any $x<0$, $\Vert \mathcal R(x,y)\Vert_{L^2(\R_y)}^2 \equiv 0$.

Altogether, we proved
\begin{multline*} \mathfrak{Q}_{h}(\psi)\geq  \int_{\R_x}(1-\eps)h^2\left(\Vert\partial_{x}(\Pi\psi)\Vert_{L^2(\R_y)}^2+\Vert\partial_{x}(\Pi^\perp\psi)\Vert_{L^2(\R_y)}^2\right)\dx x\\
+\int_{x\geq 0} -4\eps^{-1} h^2 R(x)\Vert\psi_x\Vert_{L^2(\R_y)}^2+\mu_1(x) \Vert \Pi\psi\Vert_{L^2(\R_{y})}^2+\mu_2(x) \Vert \Pi^\perp \psi\Vert_{L^2(\R_{y})}^2 \dx x ,
\end{multline*}
and the proof of Proposition~\ref{dimensional-reduction-lb} is complete since  $\langle \tilde u_{x},\partial_x \tilde u_{x}\rangle_{L^2(\R_{y})}=0$ yields
\[\Vert\partial_{x}(\Pi\psi)\Vert_{L^2(\R_y)}^2=|f'(x)|^2+|f(x)|^2\Vert \partial_x \tilde u_{x}(y)\Vert_{L^2(\R_y)}^2\geq |f'(x)|^2.\]
\end{proof}

\subsection{Reduction to model operators}
The aim of this section is to prove the following proposition.
\begin{proposition}\label{comp}
For all $f\in H^1(\R)$, we let
$$\mathfrak{Q}_{h}^{\mode1}(f)=\int_{\R} h^2|f'(x)|^2+\hat\mu_1(x)|f(x)|^2 \dx x,$$
$$\mathfrak{Q}_{h}^{\mode2}(f)=\int_{\R} h^2|f'(x)|^2+\tilde\mu_1(x)|f(x)|^2 \dx x,$$
and we denote by $\mathfrak{H}^{\mode j}_{h}$ the corresponding Friedrichs extensions. Set $M'>M$, where we denote 
\[ M=\sup_{x>0} R(x) = \sup_{x>0}\Vert\partial_{x}u_{x}\Vert^2_{L^2(\R_{y})},\]
bounded by Proposition~\ref{delta1D}. Then there exists $M_0,h_{0}>0$ such that for all $h\in(0,h_{0})$ and all $C_h\geq M_0 h$:
$$ \mathcal{N}\left(\mathfrak{H}^{\mode1}_{h}, -\frac{1}{4}- C_h-h^2 M\right)\leq\mathcal{N}\left(\mathfrak{H}_{h}, -\frac{1}{4}-C_h \right)\leq \mathcal{N}\left(\mathfrak{H}^{\mode2}_{h}, \frac{-\frac{1}{4}-C_h}{1-h}+(4 M'+1)h \right)$$
and
$$(1-h)\big\{ \lambda^{\mode2}_{n}(h)- (4M'+1) h\}\leq\lambda_{n}(h)\leq \lambda^{\mode1}_{n}(h)+h^2 M.$$
\end{proposition}
\begin{remark}
$M_0$ must be such that $M_0>4M$ and $\frac{-\frac{1}{4}-C_h}{1-h}+(4 M'+1)h<\frac{-1}{4}$, therefore one can chose $M_0=4M'+\frac34$.
\end{remark}
Let us now deal with the proof of Proposition~\ref{comp}. Lemma~\ref{op-red} suggests we introduce the following reduced operator.
\begin{notation}
For all $f\in H^1(\R)$, we let
$$\mathfrak{Q}^{\red}_{h}(f)=\int_{\R_{x}} h^2|f'(x)|^2+(\hat\mu_{1}(x)+h^2\tilde R(x))|f(x)|^2  \dx x$$
and we denote by $\mathfrak{H}^\red_{h}$ the corresponding Friedrichs extension. We define $(\lambda^\red_{n}(h),f^\red_{n})$ the $n$-th $L^2$-normalized eigenpair which exists at least for $n\in\{1,\cdots, \mathcal{N}(\mathfrak{H}^\red_{h},-\frac{1}{4})\}$.
\end{notation}
\begin{proposition}\label{red-ub}
For all $n\in\{1,\cdots,\mathcal{N}(\mathfrak{H}^\red_{h},E)\}$, with $E< -\frac{1}{4}$, and all $h>0$ the $n$-th eigenvalue of $\mathfrak{H}_{h}$ exists and satisfies:
$$\lambda_{n}(h)\leq \lambda^\red_{n}(h).$$
In particular, we have
$$\mathcal{N}(\mathfrak{H}_{h},E)\geq\mathcal{N}(\mathfrak{H}^\red_{h},E).$$
\end{proposition}

\begin{proof}
The proof relies on the introduction of suitable test functions. For any $n\in\{1,\cdots,\mathcal{N}(\mathfrak{H}^\red_{h},E)\}$, let us introduce the $n$-dimensional span
$$F_n=\underset{j\in\{1,\cdots, n\}}{\spann} f_{j}^\red(x)\tilde u_{x}(y).$$
For all $\psi\in F_n$ we have, with Lemma~\ref{op-red-bil} and noticing that the $f_{j}^\red$ are orthogonal for the bilinear form associated with $\mathfrak{Q}^{\red}_{h}$,
$$\mathfrak{Q}_{h}(\psi)\leq \lambda^\red_{n}(h)\Vert\psi\Vert^2.$$
The conclusion follows from the min-max principle and the fact that $-\frac{1}{4(1+h^2)}>-\frac{1}{4}$.
\end{proof}
We shall now analyze the reverse inequality. This is the aim of the following proposition.
\begin{proposition}\label{P.QvsQtens}
Let us consider the following quadratic form, defined on the product $H^1(\R)\times H^1(\R^2)$, by
\begin{multline*}
\mathfrak{Q}_{h}^\tens(f,\varphi)=\\
\int_{\R_{x}} (1-h)h^2|f'(x)|^2+\big(\tilde\mu_1(x)-4M h\big)|f(x)|^2 \dx x+\int_{\R^2} (1-h)h^2|\partial_{x}\varphi|^2+\big(\tilde\mu_2(x)-4M h\big)|\varphi|^2 \dx x\dx y,\\
\quad \forall (f,\varphi)\in H^1(\R)\times H^1(\R^2).
\end{multline*}
If $\mathfrak{H}^\tens_{h}$ denotes the associated operator, then we have, for all $n\geq 1$
$$\lambda_{n}(h)\geq \lambda^\tens_{n}(h).$$
\end{proposition}

\begin{proof}
We use Proposition~\ref{dimensional-reduction-lb} with $\eps=h$ and we get, for all $\psi\in\Dom(\mathfrak{Q}_{h})$, 
\begin{multline*}
\mathfrak{Q}_{h}(\psi)\geq \int_{\R_{x}} (1-h)h^2|f'|^2+\big(\tilde\mu_1(x)-4M h\big)|f|^2 \dx x\\
+\int_{\R^2} (1-h)h^2|\partial_{x}\Pi^\perp\psi|^2+\big(\tilde\mu_2(x)-4M h\big)|\Pi^\perp\psi|^2 \dx x\dx y.
\end{multline*}
Thus we have
\begin{equation}\label{tensorisation}
\mathfrak{Q}_{h}(\psi)\geq \mathfrak{Q}^\tens_{h}(\langle\psi,\tilde u_{x}\rangle_{L^2(\R_{y})},\Pi^\perp\psi),\quad \Vert\psi\Vert^2=\Vert f\Vert_{L^2(\R)}^2+\Vert\Pi^\perp\psi\Vert^2.
\end{equation}
With Notation \ref{lan} and \eqref{tensorisation} we infer
\[
   \lambda_n(h)\geq 
   \inf_{\substack{G\subset H^1(\R^2)\\ \dim G=n}} \ \sup_{\substack{
    \psi\in G}}   \frac{\mathfrak{Q}^\tens_{h}(\langle\psi,\tilde u_{x}\rangle_{L^2(\R_{y})},\Pi^\perp\psi)}{\Vert\Pi\psi\Vert^2+\Vert\Pi^\perp\psi\Vert^2} \,.
\]
Now, we define the linear injection
$$
\mathcal{J} : \left\{\begin{array}{ccc}
H^1(\R^2) &\to&  H^1(\R) \times H^1(\R^2) \\
\psi&\mapsto&(\langle\psi,\tilde u_{x}\rangle_{L^2(\R_{y})}\ ,\,\Pi^\perp\psi)
\end{array}\right..
$$
so that we have
$$ \inf_{\substack{G\subset H^1(\R^2)\\ \dim G=n}} \ \sup_{\substack{
    \psi\in G}}   \frac{\mathfrak{Q}^\tens_{h}(\langle\psi,\tilde u_{x}\rangle_{L^2(\R_{y})},\Pi^\perp\psi)}{\Vert\Pi\psi\Vert^2+\Vert\Pi^\perp\psi\Vert^2}=\inf_{\substack{\tilde G\subset \mathcal{J}(H^1(\R^2))\\ \dim \tilde G=n}} \ \sup_{\substack{(f,\varphi)\in \tilde G}}   \frac{\mathfrak{Q}^\tens_{h}(f,\varphi)}{\Vert f\Vert_{L^2(\R)}^2+\Vert\varphi\Vert^2} \,$$
and
$$\inf_{\substack{\tilde G\subset \mathcal{J}(H^1(\R^2))\\ \dim \tilde G=n}} \ \sup_{\substack{(f,\varphi)\in \tilde G}}   \frac{\mathfrak{Q}^\tens_{h}(f,\varphi)}{\Vert f\Vert_{L^2(\R)}^2+\Vert\varphi\Vert^2}\geq \inf_{\substack{\tilde G\subset H^1(\R)\times  H^1(\R^2) \\ \dim \tilde G=n}} \ \sup_{\substack{(f,\varphi)\in \tilde G}}   \frac{\mathfrak{Q}^\tens_{h}(f,\varphi)}{\Vert f\Vert_{L^2(\R)}^2+\Vert\varphi\Vert^2}.$$
We recognize the $n$-th Rayleigh quotient of $\mathfrak{H}^\tens_{h}$ and the conclusion follows.
\end{proof}

\begin{notation}
For all $f\in H^1(\R)$, we let
$$\mathfrak{Q}_{h}^\app(f)=\int_{\R} (1-h)h^2|f'(x)|^2+\big(\tilde\mu_1(x)-4M h\big)|f(x)|^2 \dx x$$
and we denote by $\mathfrak{H}^\app_{h}$ the corresponding Friedrichs extension.
\end{notation}

\begin{proposition}\label{red-lb}
For any $h>0$ and $C_h> 4Mh$, one has
$$\lambda_{n}(h)\geq \lambda^\app_{n}(h),\quad \forall n\in\left\{1,\cdots, \mathcal{N}\left(\mathfrak{H}_{h}, -\frac{1}{4}-C_h\right) \right\}$$
and
$$\mathcal{N}\left(\mathfrak{H}_{h}, -\frac{1}{4}-C_h\right)\leq \mathcal{N}\left(\mathfrak{H}^\app_{h}, -\frac{1}{4}-C_h\right).$$
\end{proposition}
\begin{proof}
Notice that for any $\varphi\in H^1(\R^2)$, one has
\[ \int_{\R^2} (1-h)h^2|\partial_{x}\varphi|^2+\big(\tilde\mu_2(x)-4M h\big)|\varphi|^2 \dx x\dx y>(-\frac14-4Mh )\Vert \varphi\Vert^2.\]
It follows that for any eigenstate of $\mathfrak H_h^\tens$ below the threshold $(-\frac14-4Mh )$ is of the form $(f,0)$, with $f$ an eigenstate of $\mathfrak{H}^\app_{h}$. In other words, one has for any $C\geq 4M$,
$$\left\{\lambda\in\sigma\left(\mathfrak{H}^\tens_{h}\right) : \lambda\leq -\frac{1}{4}-Ch\right\}=\left\{\lambda\in\sigma_{\dis}\left(\mathfrak{H}^\app_{h}\right) : \lambda\leq -\frac{1}{4}-Ch\right\},$$
and the result now follows from Proposition~\ref{P.QvsQtens}.
\end{proof}
Proposition~\ref{comp} is a direct consequence of Propositions~\ref{red-ub} and~\ref{red-lb}, and straightforward computations. In particular, we use 
\begin{align*}
\mathfrak{Q}_{h}^\app(f)&=(1-h)\int_{\R} h^2|f'(x)|^2+\frac{\tilde\mu_1(x)-4M h}{1-h}|f(x)|^2 \dx x\\
&\geq (1-h)\int_{\R} h^2|f'(x)|^2+\left(\tilde\mu_1(x)(1+h)-4M h- C h^2 \right)|f(x)|^2 \dx x\\
&\geq (1-h)\int_{\R} h^2|f'(x)|^2+\left(\tilde\mu_1(x)-(4M' +1)h \right)|f(x)|^2\dx x,
\end{align*}
which is valid for $h\in (0,h_0)$ with $h_0$ sufficiently small, $C$ sufficiently large, and any $M'>M$ (the last inequality comes from Proposition~\ref{delta1D}, item $3$). It follows
\[\mathcal{N}\left(\mathfrak{H}^\app_{h}, -\frac{1}{4}-C_h\right)\leq \mathcal{N}\left(\mathfrak{H}^{\mode2}_{h}, \frac{-\frac{1}{4}-C_h}{1-h}+(4 M'+1)h \right)\]
and for any $n\leq \mathcal{N}\left(\mathfrak{H}^{\mode2}_{h}, -\frac{1}{4}-C_h\right)$,
\[ \lambda_n^\app(h) \geq (1-h) \big\{\lambda_n^{\mode2}(h)-(4M'+1)h\big\}.\]
The condition $C_h\geq M_0 h>(4M'+\frac34)h$ ensures $\frac{-\frac{1}{4}-C_h}{1-h}+(4 M'+1)h<-\frac{1}{4}$, thus the above quantities are well-defined.

\section{Models in dimension one}\label{S4}
Thanks to Section~\ref{S3} we have reduced the spectral analysis of $\mathfrak{H}_{h}$ to the investigation of one dimensional models. This section is devoted to the proofs of Theorems~\ref{number} and~\ref{firsteigenvalues}.
\subsection{Number of bound states}
In order to prove Theorem~\ref{number} we need the following extended Weyl's asymptotics which is not completely standard (see Remark~\ref{rem-counting}).
\begin{proposition}\label{number-1D}
Let us consider $V : \R\to \R$ a piecewise Lipschitzian function with a finite number of discontinuities satisfying:
\begin{enumerate}
\item $V$ tends to $\ell_{\pm\infty}$ when $x\to\pm\infty$ with $\ell_{+\infty}\leq\ell_{-\infty}$,
\item  $\sqrt{(\ell_{+\infty}-V)_{+}}$ belongs to $L^1(\R)$.
\end{enumerate}
Consider the operator $\mathfrak{h}_{h}=-h^2\partial_{x}^2+V(x)$ and a function $(0,1)\ni h\mapsto E(h)\in(-\infty,\ell_{+\infty})$ such that one has:
\begin{enumerate}
\item for any $h\in(0,1)$, $\{x\in\R : V(x)\leq E(h)\}=[x_{\min}(E(h)),x_{\max}(E(h))]$,
\item  $h^{1/3}(x_{\max}(E(h))-x_{\min}(E(h)))\underset{h\to 0}{\to} 0$,
\item $E(h)\underset{h\to 0}{\to} E_{0}\leq \ell_{+\infty}$.
\end{enumerate}
Then we have:
$$\mathcal{N}(\mathfrak{h}_{h},E(h))\underset{h\to 0}{\sim} \frac{1}{\pi h}\int_{\R} \sqrt{(E_{0}-V)_{+}}\dx x.$$
\end{proposition}

\begin{proof}
The strategy of the proof is well-known but we recall it since the usual result does not deal with a moving threshold $E(h)$. We consider a subdivision of the real axis $(s_{j}(h^{\alpha}))_{j\in\Z}$, which contains the discontinuities of $V$, and such that there exists $c>0$, $C>0$ for which, for all $j\in\Z$ and $h>0$, $ch^\alpha\leq s_{j+1}(h^\alpha)-s_{j}(h^\alpha)\leq Ch^\alpha$, where $\alpha>0$ is to be determined. We introduce 
$$J_{\min}(h^{\alpha})=\min\{j\in\Z :  s_{j}(h^\alpha)\geq x_{\min}(E(h))\},$$
$$J_{\max}(h^{\alpha})=\max\{j\in\Z :  s_{j}(h^\alpha)\leq x_{\max}(E(h))\}.$$
For $j\in\Z$ we may introduce the Dirichlet (resp. Neumann) realization on $(s_{j}(h^{\alpha}), s_{j+1}(h^\alpha))$ of $-h^2\partial_{x}^2+V(x)$ denoted by $\mathfrak{h}_{h,j}^\Dir$ (resp. $\mathfrak{h}_{h,j}^\Neu$). The so-called Dirichlet-Neumann bracketing (see~\cite[Chapter XIII, Section 15]{ReSi78}) implies:
$$\sum_{j=J_{\min}(h^\alpha)}^{J_{\max}(h^\alpha)} \mathcal{N}(\mathfrak{h}^\Dir_{h,j},E(h)) \leq\mathcal{N}(\mathfrak{h}_{h},E(h))\leq \sum_{j=J_{\min}(h^\alpha)-1}^{J_{\max}(h^\alpha)+1} \mathcal{N}(\mathfrak{h}^\Neu_{h,j},E(h)).$$
Let us estimate $\mathcal{N}(\mathfrak{h}^\Dir_{h,j},E(h))$. If $\mathfrak{q}_{h,j}^\Dir$ denotes the quadratic form of $\mathfrak{h}^\Dir_{h,j}$, we have:
$$\mathfrak{q}_{h,j}^\Dir(\psi)\leq \int_{s_{j}(h^\alpha)}^{s_{j+1}(h^\alpha)} h^2|\psi'(x)|^2+V_{j,\sup,h}|\psi(x)|^2\dx x,\quad\forall\psi\in\mathcal{C}^\infty_{0}((s_{j}(h^{\alpha}), s_{j+1}(h^\alpha))),$$
where 
$$V_{j,\sup,h}=\sup_{x\in(s_{j}(h^{\alpha}), s_{j+1}(h^\alpha))} V(x).$$
We infer that
$$\mathcal{N}(\mathfrak{h}^\Dir_{h,j},E(h))\geq\#\left\{n\geq 1 : n\leq \frac{1}{\pi h}(s_{j+1}(h^{\alpha})-s_{j}(h^\alpha))\sqrt{\left(E(h)-V_{j,\sup,h}\right)_{+}}\right\}$$
so that:
$$\mathcal{N}(\mathfrak{h}^\Dir_{h,j},E(h))\geq  \frac{1}{\pi h}(s_{j+1}(h^{\alpha})-s_{j}(h^\alpha))\sqrt{\left(E(h)-V_{j,\sup,h}\right)_{+}}-1$$
and thus:
\begin{multline*}
\sum_{j=J_{\min}(h^\alpha)}^{J_{\max}(h^\alpha)} \mathcal{N}(\mathfrak{h}^\Dir_{h,j},E(h))\geq \\
\frac{1}{\pi h}\sum_{j=J_{\min}(h^\alpha)}^{J_{\max}(h^\alpha)}(s_{j+1}(h^{\alpha})-s_{j}(h^\alpha))\sqrt{\left(E(h)-V_{j,\sup,h}\right)_{+}} -(J_{\max}(h^\alpha)-J_{\min}(h^\alpha)+1).
\end{multline*}
Let us consider the function 
$$f_{h}(x)=\sqrt{\left(E(h)-V(x)\right)_{+}}$$
and analyze
\begin{multline*}
\left|\sum_{j=J_{\min}(h^\alpha)}^{J_{\max}(h^\alpha)}(s_{j+1}(h^{\alpha})-s_{j}(h^\alpha))\sqrt{\left(E(h)-V_{j,\sup,h}\right)_{+}}       -\int_{\R} f_{h}(x)\dx x\right|\\
\leq \left|\sum_{j=J_{\min}(h^\alpha)}^{J_{\max}(h^\alpha)} \int_{s_{j}(h^\alpha)}^{s_{j+1}(h^\alpha)} \sqrt{\left(E(h)-V_{j,\sup,h}\right)_{+}}-f_{h}(x)\dx x\right|\\
+\int_{s_{J_{\max}}(h^\alpha)}^{x_{\max}(E(h))} f_{h}(x)\dx x+\int_{x_{\min}(E(h))}^{s_{J_{\min}(h^\alpha)}} f_{h}(x)\dx x\\
\leq  \left|\sum_{j=J_{\min}(h^\alpha)}^{J_{\max}(h^\alpha)} \int_{s_{j}(h^\alpha)}^{s_{j+1}(h^\alpha)} \sqrt{\left(E(h)-V_{j,\sup,h}\right)_{+}}-f_{h}(x)\dx x\right|+\tilde C h^\alpha.
\end{multline*}
Using the trivial inequality $|\sqrt{a_{+}}-\sqrt{b_{+}}|\leq \sqrt{|a-b|}$, we notice that 
$$|f_{h}(x)-\sqrt{\left(E(h)-V_{j,\sup,h}\right)_{+}}|\leq \sqrt{|V(x)-V_{j,\sup,h}|}.$$
Since $V$ is Lipschitzian on $(s_{j}(h^{\alpha}), s_{j+1}(h^\alpha))$, we get:
$$\left|\sum_{j=J_{\min}(h^\alpha)}^{J_{\max}(h^\alpha)} \int_{s_{j}(h^\alpha)}^{s_{j+1}(h^\alpha)}\sqrt{\left(E(h)-V_{j,\sup,h}\right)_{+}}-f_{h}(x)\dx x\right|\leq (J_{\max}(h^\alpha)-J_{\min}(h^\alpha)+1) \tilde Ch^\alpha h^{\alpha/2}.$$
This leads to the optimal choice $\alpha=\frac{2}{3}$ and we get the lower bound:
$$\sum_{j=J_{\min}(h^{2/3})}^{J_{\max}(h^{2/3})} \mathcal{N}(\mathfrak{h}^\Dir_{h,j},E(h))\geq \frac{1}{\pi h}\left( \int_{\R} f_{h}(x)\dx x-\tilde Ch (J_{\max}(h^{2/3})-J_{\min}(h^{2/3})+1)-\tilde C h^{2/3}\right).$$
Therefore we infer
$$\mathcal{N}(\mathfrak{h}_{h},E(h))\geq \frac{1}{\pi h}\left(\int_{\R} f_{h}(x)\dx x-\tilde Ch^{1/3}(x_{\max}(E(h))-x_{\min}(E(h))-\tilde C h^{2/3}\right).$$
We notice that: $f_{h}(x)\leq \sqrt{(\ell_{+\infty}-V(x))_{+}}$ so that we can apply the dominate convergence theorem.
We can deal with the Neumann realizations in the same way.
\end{proof}
\begin{remark}\label{rem-counting}
Classical results (see~\cite{ReSi78, Ro87, DiSj99, Z13}) impose a fixed security distance below the edge of the essential spectrum ($E(h)=E_0<l_{+\infty}$) or deal with non-negative potentials, $V$, with compact support. Both these cases are recovered by Proposition~\ref{number-1D}. In our result, the maximal threshold for which one can ensure that the semiclassical behavior of the counting function holds is dictated by the convergence rate of the potential towards its limit at infinity, through the assumption
\[ h^{1/3}(x_{\max}(E(h))-x_{\min}(E(h)))\underset{h\to 0}{\to} 0. \]
More precisely, assume that $l_{-\infty}>l_{+\infty}$ so that $x_{\min}(E(h))\geq x_{\min}(l_{+\infty})$ is uniformly bounded for $E(h)$ in a neighborhood of $l_{+\infty}$. Then
\begin{itemize}
\item If $l_{+\infty}-V(x)\leq C  x^{-\gamma}$ for any $x\geq x_0$ and given $x_0,C>0$ and $\gamma>2$, then one can choose $E(h)=l_{+\infty}-C h^\rho$ and $x_{\max}(E(h))\leq h^{-\rho/\gamma}$, provided $\rho<\gamma/3$.
\item If $l_{+\infty}-V(x)\leq C_1 \exp(-C_2  x)$ for any $x\geq x_0$ and given $x_0,C_1,C_2>0$, then one can choose $E(h)=l_{+\infty}-C_1\exp(C_2 h^{-1/3}\times o(h))$ and the assumption is satisfied.
\end{itemize}
\end{remark}
\paragraph{Proof of Theorem~\ref{number}}
In order to prove Theorem~\ref{number} we apply Proposition~\ref{comp} with $C_{h}=C(h)$.
Then we apply Proposition~\ref{number-1D} to the operators $\mathfrak{H}^{\mode j}_{h}$. Increasing $M_{0}$ if necessary, we have $E(h)\leq -\frac14-C h$ with any $C>0$ and therefore the assumptions of Proposition~\ref{number-1D} are satisfied. Indeed from Proposition~\ref{delta1D}, $\tilde\mu_{1}$ and $\hat\mu_{1}$ converge exponentially to $-\frac{1}{4}$ as $x\to \infty$, and $\tilde\mu_{1},\hat\mu_{1}>-\frac{1}{4}$ for $x<0$.

\subsection{Low lying spectrum}
Let us now deal with the proofs of Theorem~\ref{firsteigenvalues} and~\ref{tensorisation-x=0}.
\subsubsection{Proof of Theorem~\ref{firsteigenvalues}}
The following proposition provides the asymptotics of the lowest eigenvalues of the models $\mathfrak{H}^{\mode j}_{h}$ and is a direct consequence of the analysis of~\cite[Section 3]{DauRay12}.
\begin{proposition}\label{low-lying1D}
For $j=1,2$ and for all $n\geq 1$ we have:
$$\lambda^{\mode j}_{n}(h)=-1+2^{2/3} z_{\Ai}(n)h^{2/3}+O(h).$$
\end{proposition}
With Proposition~\ref{comp} this implies Theorem~\ref{firsteigenvalues}.
In fact, it is possible to establish some localization properties of the first eigenfunctions of $\mathfrak{H}_{h}$.
\begin{proposition}\label{Agmon}
Let $\lambda\in(-1,0)$ and $\delta\in(0,1)$. For all $h>0$ and all eigenpairs $(\lambda,\psi)$ of $\mathfrak{H}_{h}$, we have
\begin{equation}\label{Agmon1}
\int_{\R_{y}}\int_{-\infty}^0 e^{2(1-\delta)\sqrt{-\lambda} h^{-1}|x|}|\psi|^2\dx x\dx y\leq \frac{1}{(-\lambda)\delta^2}\Vert\psi\Vert^2.
\end{equation}
Moreover, for all $C_{0}>0$, there exist $h_{0}>0$, $C>0$ and $\eps_{0}>0$ such that for all $h\in(0,h_{0})$ and all eigenpairs $(\lambda,\psi)$ such that $\lambda\leq -1+C_{0}h^{2/3}$, we have
\begin{equation}\label{Agmon2}
\int_{\R_{y}}\int_{0}^{+\infty} e^{2\eps_{0} h^{-2/3}|x|}|\psi|^2\dx x\dx y\leq C\Vert\psi\Vert^2.
\end{equation}
\end{proposition}
\begin{proof}
This is a consequence of Proposition~\ref{lb-rQ} and of Agmon type estimates inherited from the one dimensional operator $-h^2 -\partial_{x}^2+\tilde\mu_{1}(x)$ (see~\cite{DauRay12} and also the original references~\cite{Agmon82, Hel88}).
\end{proof}

\subsubsection{Proof of Theorem~\ref{tensorisation-x=0}}
Let us consider an eigenpair $(\lambda,\psi)$ such that $\lambda\leq -1+C_{0}h^{2/3}$. We can write
$$\mathfrak{Q}_{h}(\psi)=\lambda\Vert\psi\Vert^2$$
and
$$\mathfrak{Q}_{h}(\psi)=\mathfrak{Q}_{h,+}(\psi)+\mathfrak{Q}_{h,-}(\psi),$$
where
$$\mathfrak{Q}_{h,-}(\psi)=\int_{\R^-_{x}\times\R_{y}}h^2|\partial_{x}\psi|^2+|\partial_{y}\psi|^2 \dx x \dx y,$$
$$\mathfrak{Q}_{h,+}(\psi)=\int_{\R^+_{x}\times\R_{y}}h^2|\partial_{x}\psi|^2+|\partial_{y}\psi|^2 \dx x \dx y-\int_{\R^+} |\psi(x,x)|^2\dx x-\int_{\R^{+}} |\psi(x,-x)|^2\dx x.$$
We infer that
$$\mathfrak{Q}_{h,+}(\psi)+\mathfrak{Q}_{h,-}(\psi)+\int_{\R^+_{x}\times\R_{y}}|\psi|^2\dx x\dx y+\int_{\R^-_{x}\times\R_{y}}|\psi|^2\dx x\dx y\leq C_{0}h^{2/3}\Vert\psi\Vert^2.$$
By Point \ref{5} of Proposition \ref{delta1D}, we deduce that
\begin{equation}\label{Q-<}
0\leq\mathfrak{Q}_{h,-}(\psi)+\int_{\R^-_{x}\times\R_{y}}|\psi|^2\dx x\dx y\leq C_{0}h^{2/3}\Vert\psi\Vert^2
\end{equation}
and
\begin{equation}\label{Q+<}
0\leq\mathfrak{Q}_{h,+}(\psi)+\int_{\R^+_{x}\times\R_{y}}|\psi|^2\dx x\dx y\leq C_{0}h^{2/3}\Vert\psi\Vert^2.
\end{equation}
We recall the points~\ref{1} and \ref{4} of Proposition~\ref{delta1D} to deduce from~\eqref{Q+<} that
$$\int_{\R^+_{x}\times\R_{y}} h^2 |\partial_{x}\psi|^2 \dx x \dx y+\int_{\R^+}\mathfrak{q}_{x}(\psi_{x}) \dx x-\int_{\R^+_{x}\times\R_{y}}\mu_{1}(x)|\psi|^2\dx x\dx y\leq C_{0}h^{2/3}\Vert\psi\Vert^2$$
where we recall that
$$\mathfrak{q}_{x}(\psi_{x})=\int_{\R_{y}} |\partial_{y}\psi_{x}|^2 \dx y- |\psi(x,x)|^2- |\psi(x,-x)|^2.$$
We have
$$\mathfrak{q}_{x}(\psi_{x})-\mu_{1}(x)\Vert\psi\Vert^2_{L^2(\R_{y})}=\mathfrak{q}_{x}(\psi-\Pi\psi)-\mu_{1}(x)\Vert\psi-\Pi\psi\Vert^2_{L^2(\R_{y})}$$
and then, due to the min-max principle,
$$\mathfrak{q}_{x}(\psi_{x}-(\Pi\psi)_{x})\geq\mu_{2}(x)\Vert\psi-\Pi\psi\Vert^2_{L^2(\R_{y})}.$$
We get
$$\int_{\R^+_{x}\times\R_{y}} (\mu_{2}(x)-\mu_{1}(x))|\psi-\Pi\psi|^2\dx x\dx y\leq C_{0}h^{2/3}\Vert\psi\Vert^2.$$
Due to the simplicity of $\mu_{1}$, we can find $\eps_{0}>0$ such that for $x\in[0,1]$ we have
$$\mu_{2}(x)-\mu_{1}(x)\geq \eps_{0}.$$
Then for $x\geq 1$ we use the estimates of Agmon~\eqref{Agmon2} and the boundedness of the $\mu_{j}$ to get
$$\int_{\R_{y}}\int_{x>1} (\mu_{2}(x)-\mu_{1}(x))|\psi-\Pi\psi|^2\dx x\dx y\leq C\int_{\R_{y}}\int_{x>1}|\psi|^2\dx x \dx y\leq Ce^{-2\eps_{0}h^{-2/3}}\Vert\psi\Vert^2,$$
where we have used
$$\|\Pi^\perp\psi\|_{L^2(\R_{y})}^2\leq \|\psi\|_{L^2(\R_{y})}^2.$$
We deduce that
$$\int_{\R^+_{x}\times\R_{y}} |\psi-\Pi\psi|^2\dx x\dx y\leq Ch^{2/3}\Vert\psi\Vert^2.$$
We have proved (it follows from the point~\ref{6} of Proposition~\ref{delta1D}) that  the application $[0,+\infty)\ni x\mapsto \Pi_{x}=\langle\cdot,\tilde u_{x}\rangle_{L^2(\R_{y})}\tilde u_{x}\in \mathcal{L}_{c}(L^2(\R_{y}),L^2(\R_{y}))$ is Lipschitzian (with Lipschitz constant $K>0$) so that
$$\Vert(\Pi_{x}-\Pi_{0})\psi_{x}\Vert_{L^2(\R_{y})}\leq K|x|\Vert\psi\Vert_{L^2(\R_{y})}.$$
Let us now consider for instance $\eta\in\left(0,\frac{1}{100}\right)$. We infer that
$$\int_{\R_{y}}\int_{0<x<h^{2/3-\eta}} |\Pi_{x}\psi_{x}-\Pi_{0}\psi|^2\dx x \dx y\leq K^2 h^{4/3-2\eta}\Vert\psi\Vert^2.$$
Thanks to the estimates of Agmon, we have
$$\int_{\R_{y}}\int_{x>h^{2/3-\eta}} |\Pi_{x}\psi_{x}-\Pi_{0}\psi|^2\dx x\dx y\leq Ce^{-2\eps_{0}h^{-2/3}}\Vert\psi\Vert^2.$$
We deduce that
$$\int_{\R^+_{x}\times\R_{y}} |\psi-\Pi_{0}\psi|^2\dx x\dx y\leq Ch^{2/3}\Vert\psi\Vert^2$$
and, since
$$\int_{\R^-_{x}\times\R_{y}}  |\psi-\Pi_{0}\psi|^2\dx x\dx y\leq \int_{\R^-_{x}\times\R_{y}}|\psi|^2\dx x\dx y,$$
the conclusion follows from~\eqref{Q-<}.

\paragraph{Acknowledgments} The second author would like to thank K. Pankrashkin which gave him the initial impulse to study the broken $\delta$-interactions.


\end{document}